\documentclass[12pt]{amsart}
\usepackage{amssymb}
\usepackage{calrsfs}
\usepackage[bookmarks=false]{hyperref}

\newtheorem{theorem}{Theorem}[section]
\newtheorem{lemma}[theorem]{Lemma}
\newtheorem{proposition}[theorem]{Proposition}
\newtheorem{corollary}[theorem]{Corollary}

\newtheorem*{theorem:Main}{Theorem~\ref{T:Main}}

\theoremstyle{definition}

\newtheorem{example}[theorem]{Example}
\newtheorem{remark}[theorem]{Remark}
\newtheorem*{acknowledgments}{Acknowledgments}

\theoremstyle{remark}

\newcommand{\FF}{\mathbb{F}}
\newcommand{\ZZ}{\mathbb{Z}}
\newcommand{\QQ}{\mathbb{Q}}

\newcommand{\GG}{\mathbb{G}}
\newcommand{\II}{\mathbb{I}}

\newcommand{\CC}{\mathbb{C}}

\newcommand{\be}{\mathbf{e}}

\newcommand{\bu}{\mathbf{u}}

\newcommand{\cE}{\mathcal{E}}
\newcommand{\cL}{\mathcal{L}}

\DeclareMathOperator{\Log}{Log}

\DeclareMathOperator{\End}{End}

\DeclareMathOperator{\Gal}{Gal}

\DeclareMathOperator{\ord}{ord}

\newcommand{\oK}{\mkern2.5mu\overline{\mkern-2.5mu K}}

\newcommand{\ophi}{\mkern2.5mu\overline{\mkern-2.5mu \phi}}

\newcommand{\sep}{\mathrm{sep}}
\newcommand{\tor}{\mathrm{tor}}

\newcommand{\tpi}{\widetilde{\pi}}

\newcommand{\Ga}{\GG_{\mathrm{a}}}

\newcommand{\power}[2]{{#1 [[ #2 ]]}}
\newcommand{\laurent}[2]{{#1 (( #2 ))}}
\newcommand{\brac}[2]{\genfrac{\{}{\}}{0pt}{}{#1}{#2}}
\newcommand{\norm}[1]{\lVert #1 \rVert}
\newcommand{\smod}[1]{{\, (\mathrm{mod}\, #1)}}

\begin{document}

\title[Log-algebraic identities on Drinfeld modules and special $L$-values]{Log-algebraic identities on Drinfeld modules \\ and special $L$-values}

\author{Chieh-Yu Chang}
\address{Department of Mathematics, National Tsing Hua University, Hsinchu City 30042, Taiwan R.O.C.}
\email{cychang@math.nthu.edu.tw}

\author{Ahmad El-Guindy}
\address{Current address: Science Program, Texas A{\&}M University in Qatar,
Doha, Qatar}
\address{Permanent address: Department of Mathematics, Faculty of Science,
Cairo University, Giza, Egypt 12613}
\email{a.elguindy@gmail.com}

\author{Matthew A. Papanikolas}
\address{Department of Mathematics, Texas A{\&}M University, College Station,
TX 77843, U.S.A.}
\email{papanikolas@tamu.edu}

\thanks{The first author was partially supported by  MOST Grant
  102-2115-M-007-013-MY5.  The third author was partially supported by NSF Grant DMS-1501362}

\subjclass[2010]{Primary 11G09; Secondary 11M38, 11J93}

\date{November 5, 2017}

\begin{abstract}
We formulate and prove a log-algebraicity theorem for arbitrary rank Drinfeld modules defined over the polynomial ring $\mathbb{F}_q[\theta]$.  This generalizes results of Anderson for the rank one case.  As an application we show that certain special values of Goss $L$-functions are linear forms in Drinfeld logarithms and are transcendental.
\end{abstract}

\keywords{Drinfeld modules, log-algebraicity, Taelman's formula, Goss $L$-functions}

\maketitle

\section{Introduction} \label{S:Intro}

In \cite{And94} and \cite{And96}, Anderson introduced the notion of log-algebraicity for rank one Drinfeld modules, inspired by earlier special cases of Thakur~\cite{Thakur92}.  He demonstrated that these power series identities could be used to express values of Goss zeta and $L$-functions at $s=1$ as linear combinations of logarithms.  In the present paper we investigate log-algebraic identities for Drinfeld modules of arbitrary rank over the polynomial ring $\FF_q[\theta]$, and we prove that particular special values of Goss $L$-functions can be expressed in terms of linear combinations of Drinfeld logarithms, thus recovering and extending previous results of Taelman~\cite{Taelman12}.

We let $\FF_q$ be a field with $q$ elements and let $A = \FF_q[\theta]$ be the polynomial ring in a variable~$\theta$.  We take
\[
  A_+ = \{ a \in A \mid \textup{$a$ monic} \}, \quad
  A_{i+} = \{ a \in A_+ \mid \deg a = i \}.
\]
We let $\infty$ denote the infinite place of the fraction field $K = \FF_q(\theta)$ with valuation given by $\ord_\infty = -\deg_{\theta}$ and absolute value normalized by $|\theta|_{\infty} = q$.  We take $K_\infty = \laurent{\FF_q}{1/\theta}$ for the completion of $K$ at $\infty$, and we let~$\CC_\infty$ denote the completion of an algebraic closure of $K_\infty$. Let $\oK$ be the algebraic closure of $K$ in $\CC_{\infty}$.

For simplicity, we recall Anderson's result for the case of Carlitz module, and let us consider the special values of Goss $L$-functions for Dirichlet characters at $s=1$. Fixing an irreducible polynomial $\wp\in A_{+}$ and a Dirichlet character $\chi : A \to \overline{\FF}_q$ modulo $\wp$, we put
\[
L(1,\chi):=\sum_{a\in A_{+}} \frac{\chi(a) }{ a}\in \CC_{\infty}^{\times}.
\]
This special value plays the analogous role of the value at $s=1$ of the classical $L$-series for a Dirichlet character modulo a prime $p$, which is known to be a $\overline{\QQ}$-linear combination of logarithms at certain circular units when the given Dirichlet character is nontrivial (see \cite[p.~37]{Washington}).  To study such $L$-values, Anderson~\cite{And96} introduced the following power series as a kind of deformation of the $L$-value,
\[
  \cL_{C}(\beta,z):=\sum_{a\in A_{+}}\frac{a\star \beta }{a}z^{q^{\deg a}}\in \power{K[x]}{z},
\]
where $x$ and $z$ are new independent variables, $C$ denotes the Carlitz $A$-module defined in~\eqref{E:DefCarlitz}, $\beta \in K[x]$, and $a \star \beta \in K[x]$ is defined by Anderson's $\star$-operation given in \eqref{E:stardef}.

Let $\exp_{C}(z) \in \power{K}{z}$ be the Carlitz exponential function given in \eqref{E:expC}. Anderson's log-algebraicity theorem~\cite[Thm.~3]{And96} asserts that for $\beta \in A[x]$,
\begin{equation}\label{E:expCofL}
\exp_{C}\bigl( \cL_{C}(\beta, z) \bigr)\in A[x,z].
\end{equation}
As an important consequence, by the analogue of the Hermite-Lindemann theorem of Yu~\cite{Yu86},  nonzero values of $\cL_{C}(\beta,z)$ that specialize $x$ and $z$ at elements of $\oK$ are transcendental over~$K$.  Using certain specializations of \eqref{E:expCofL}, Anderson derived an explicit formula for $L(1,\chi)$ in terms of the Carlitz logarithm at certain explicit algebraic points (which are analogues of the classical circular units).  It follows that each $L(1,\chi)$ is transcendental over $K$ by Yu's analogue~\cite{Yu97} of Baker's theorem on linear forms in logarithms.

In this paper we consider Drinfeld $A$-modules of generic characteristic, which are defined over $A$ and have arbitrary rank.  Given Anderson's formulation above it is not immediately clear how to formulate log-algebraicity results for higher rank Drinfeld modules.  However, inspired by Taelman's work on special $L$-values~\cite{Taelman09}, \cite{Taelman10}, \cite{Taelman12}, we succeed in discovering the right point of view.

Fixing a Drinfeld $A$-module $\phi$ that is defined over $A$ (see \eqref{E:phidef}), Taelman~\cite{Taelman12} defined its associated special $L$-value $L(\phi/A)$, as in \eqref{E:TaelmanL}. One notes that $L(\phi/ A)$ is identical to the special value at $s=0$ of the Goss $L$-function denoted by $L(\phi^{\vee},0)$, arising from the compatible system of the Galois representations on the dual of the Tate module of $\phi$, when the Drinfeld module~$\phi$ has everywhere good reduction.  Along these lines, we define the Goss $L$-function $L(\phi^{\vee},s)$ in \eqref{E:Lphidual}, making a particularly suitable choice of the local factors at the bad primes of $\phi$.  In order to construct a log-algebraicity result via a twisted harmonic sum over $A_+$, we shift $s$ by $1$ in $L(\phi^{\vee},s)$ and obtain the Dirichlet series,
\[
  L(\phi^{\vee},s-1) = \sum_{a \in A_+} \frac{\mu(a)}{a^s}
\]
(see \eqref{E:mudef}). We then form the power series
\[
  \cL_{\phi}(\beta,z):=\sum_{a\in A_{+}} \frac{\mu(a)(a\star \beta)(x) }{a}z^{q^{\deg a}} \in \power{K[x]}{z}
\]
for $\beta\in A[x]$, and our main theorem is as follows.

\begin{theorem:Main}
For any $\beta \in A[x]$, the power series
\[
  \cE_{\phi}(\beta,z) := \exp_{\phi} \bigl( \cL_{\phi}(\beta,z) \bigr) \in \power{K[x]}{z},
\]
is in fact in $A[x,z]$.
\end{theorem:Main}

Fixing a Dirichlet character $\chi$ modulo a prime $\wp \in A_{+}$, we consider the following $L$-value twisted by $\chi$:
\[
L(\phi^{\vee},\chi,0):=\sum_{a\in A_{+}}\frac{\mu(a)\chi(a)}{a}\in \CC_{\infty}^{\times}.
\]
As in the case of $L(1,\chi)$, we use Theorem~\ref{T:Main} to demonstrate that
\begin{itemize}
\item $L(\phi^{\vee},\chi,0)$ is a $\oK$-linear combination of Drinfeld logarithms at certain explicit algebraic points (see Corollary~\ref{C:L-log});
\item $L(\phi^{\vee},\chi,0)$ is transcendental over $K$ (see Corollary~\ref{C:Trans-L}).
\end{itemize}
Our proof of Theorem~\ref{T:Main} is rooted in Anderson's strategy~\cite{And96} as follows:
\begin{itemize}
\item We prove the integrality result that the series $\cE_{\phi}(\beta,z)$ has coefficients in $A[x]$ (see Theorem~\ref{T:Integrality}).
\item We use $\infty$-adic estimates to show that $\cE_{\phi}(\beta,z)$ is indeed a polynomial in $z$ (see Theorem~\ref{T:Estimates}).
\end{itemize}

Although the general outline of Anderson's method is robust and, as we will see, can be used for higher ranks, the coefficients of the Dirichlet series $L(\phi^{\vee},s-1)$ given by the multiplicative function $\mu : A_+ \to A$ vary unpredictably and require careful accounting.  By contrast, in the case of the Carlitz module the coefficients are identically~$1$.  To better understand $\mu$, we investigate properties of the characteristic polynomial of Frobenius acting on the Tate module for the reduction of $\phi$ modulo an irreducible $f \in A_+$ and prove new congruence results modulo $f$ for the coefficients of the polynomial $\phi_f(x) \in A[x]$ defining the $f$-operation of $\phi$ (see Lemmas~\ref{L:muprops} and~\ref{L:braccongs}).

In recent years much research has been conducted on extensions of Anderson's log-algebrai\-city theorem to various settings.  Angl\`es, Pellarin, Taelman, and Tavares Ribeiro~\cite{AnglesPellarinTavares16}, \cite{AnglesPellarinTavaresTAMS}, \cite{AnglesTaelman15}, \cite{AnglesTavares17}, have investigated multivariable versions of Anderson's theorem for the Carlitz module and its tensor powers with values in Tate algebras, using new versions of Taelman's class number formula, and have studied modules of special points.  These results were then generalized to rank one Drinfeld modules over more general rings by Angl\`es, Ngo Dac, and Tavares Ribeiro~\cite{AnglesNgoDacTavares17}.  Work of Green and the third author~\cite{GreenP} provides explicit formulas for log-algebraic identities for rank one Drinfeld modules over coordinate rings of elliptic curves.  Log-algebraicity on tensor powers of the Carlitz module is investigated in~\cite{PLogAlg}.  For the most part (save some results in~\cite{AnglesTavares17}, see below) these results lie in the realm of rank one, and one of the underlying goals of the present work is understand these phenomena in higher ranks.

It is important to mention that Angl\`es and Tavares Ribeiro~\cite{AnglesTavares17} considered the $z$-deformed Drinfeld module $\tilde{\phi}$ of a given Drinfeld $A$-module $\phi$ defined over the ring of integers of a finite extension of $K$, and they established an `equivariant' log-algebraicity result for $\tilde{\phi}$ in \cite[Thm.~2]{AnglesTavares17} (we note that the exponential of $\tilde{\phi}$ is not the same as the exponential of $\phi$, but they are closely related).  The proof of their Theorem~2 is based on equivariant class module formulas~\cite[Prop.~4]{AnglesTavares17}, and it differs from Anderson's original methods that we study in this paper. We thank Angl\`es for clarifying these and related issues.  We also thank him for sharing his ideas with us about the possibility of the connection between~\cite[Thm.~2]{AnglesTavares17} and our Theorem~\ref{T:Main}, which would require additional work beyond the scope of this paper.  We further refer the reader to~\cite[Cor.~3]{AnglesTavares17}, where the authors use~\cite[Thm.~2]{AnglesTavares17} to recover Anderson's original log-algebraicity result in the case of the Carlitz module.

The paper proceeds as follows.  In \S\ref{S:Notation} we review the fundamental definitions of Drinfeld $A$-modules and recall aspects of Taelman's special $L$-value formulas.  In \S\ref{S:Charpoly} we discuss the reduction of Drinfeld modules and the construction of Goss $L$-series via characteristic polynomials of Frobenius.  In \S\ref{S:LogAlg} we state the main log-algebraicity result (Theorem~\ref{T:Main}) and explore applications to special $L$-values (Corollaries~\ref{C:L-log} and~\ref{C:Trans-L}).  The main proof is contained in \S\ref{S:Integrality} (integrality estimates) and \S\ref{S:Degrees} (degree estimates).  We conclude with examples in \S\ref{S:Examples}.

\begin{acknowledgments}
We thank Bruno Angl\`{e}s for several discussions about the contents of this paper, including clarifying the connections with his joint work with Tavares Ribeiro.  We also thank the referee for constructive comments that greatly improved the quality of the paper.
\end{acknowledgments}

\section{Notation and setting} \label{S:Notation}

We continue with the notation given in the introduction. Let $\tau : \CC_\infty \to \CC_\infty$ denote the $q$-th power Frobenius endomorphism. For any $\FF_q$-subalgebra $R \subseteq \CC_\infty$ we take $R[\tau]$ to be the ring of twisted polynomials in $\tau$, which are subject to the relation $\tau c = c^q \tau$ for $c \in R$.

We fix throughout a Drinfeld module $\phi : A \to A[\tau]$ of rank~$r \geq 1$ defined by
\begin{equation} \label{E:phidef}
  \phi_\theta = \theta + \kappa_1 \tau + \dots + \kappa_r \tau^r, \quad \kappa_r \neq 0.
\end{equation}
For $a \in A$, we set
\begin{equation} \label{E:brac}
  \phi_a = \sum_{k = 0}^{r \deg a} \brac{a}{k} \tau^k,
\end{equation}
and if $k < 0$ or $k > r \deg a$ we set $\brac{a}{k} = 0$, thus defining $\brac{a}{k} \in A$ for all $a \in A$ and $k \in \ZZ$.  For an $A$-algebra $L$, we let $\phi(L)$ denote a copy of $L$ together with the $A$-module structure induced by $\phi$.  See \cite[Ch.~4]{Goss}, \cite[Ch.~2]{Thakur} for more information on Drinfeld modules.  For convenience we call the additive group $\Ga$ with scalar $A$-action a Drinfeld module of rank~$0$.

The exponential function $\exp_\phi(z) = \sum_{i \geq 0} \alpha_i z^{q^i} \in \power{K}{z}$, with $\alpha_0=1$, is defined by the condition
\begin{equation} \label{E:expfneq}
  \exp_\phi(az) = \phi_a \bigl( \exp_\phi(z) \bigr), \quad \forall\, a \in A,
\end{equation}
and it induces an entire, surjective, $\FF_q$-linear function $\exp_\phi : \CC_\infty \to \CC_\infty$.  Its formal inverse is the logarithm function $\log_\phi(z) = \sum_{i \geq 0} \beta_i z^{q^i} \in \power{K}{z}$, which has a finite radius of convergence in $\CC_\infty$.  Formulas for the coefficients $\alpha_i$, $\beta_i$ can be found in~\cite[\S 3]{EP13}.

An irreducible polynomial $f \in A_+$ of degree $d$ is said to be a prime of good reduction for $\phi$ if $f \nmid \kappa_r$, and otherwise it is a prime of bad reduction.  We let $\FF_f$ denote the field $A/(f)$ with induced $A$-module structure, and we let $\ophi : A \to \FF_f[\tau]$ denote the reduction of $\phi$ modulo $f$,
\begin{equation} \label{E:ophi}
  \ophi_\theta = \bar{\theta} + \bar{\kappa}_1 \tau + \dots + \bar{\kappa}_{r_0} \tau^{r_0}, \quad \bar{\kappa}_{r_0} \neq 0,
\end{equation}
where for $a \in A$ we take $\bar{a} \in \FF_f$ to be its reduction modulo $f$.  We note that $\ophi$ is a Drinfeld module with characteristic $(f)$ of rank $r_0$ with $0 \leq r_0 \leq r$.  The dependence of $r_0$ on $f$ is implied and should not lead to too much confusion.  If $L / \FF_f$ is a field extension, we take $\ophi(L)$ to be $L$ with the $A$-module structure induced by $\ophi$.  For more information on Drinfeld modules over finite fields, see Gekeler~\cite{Gekeler91}.

For a finite $A$-module $M$ with
\[
M \cong A/(f_1) \oplus \cdots \oplus A/(f_n), \quad f_1, \dots, f_n \in A_+,
\]
as in Taelman~\cite[\S 1]{Taelman12} we set
\[
  |M| := f_1 \cdots f_n,
\]
which is independent of the direct sum decomposition chosen and serves as the analogue of the cardinality of a finite abelian group.  The ideal generated by $|M|$ in $A$ is the Fitting ideal of $M$ and also the Euler-Poincar\'{e} characteristic of $M$ defined in Gekeler~\cite[\S 3]{Gekeler91}.

We thus can define Taelman's $L$-value by the infinite product over irreducible $f \in A_+$,
\begin{equation} \label{E:TaelmanL}
  L(\phi/A) = \prod_f \frac{|\FF_f|}{|\ophi(\FF_f)|},
\end{equation}
for which Taelman proved the following result.

\begin{theorem}[{Taelman~\cite[Thm.~1]{Taelman12}}] \label{T:Taelman}
We have
\[
  \exp_\phi \bigl( L(\phi/A) \bigr) \in A.
\]
\end{theorem}

\begin{remark}
Taelman's theorem is in fact more general and precise than what we state here.  He proves that $L(\phi/A)$ can be expressed as the product of the order of a certain finite $A$-module (the class module) and the logarithm of a special point in $\phi(A)$.  Furthermore, he works over arbitrary extensions of $A$ as well.  However, as we saw in \S\ref{S:Intro}, Theorem~\ref{T:Main} provides as a corollary a new proof of Theorem~\ref{T:Taelman} as stated here.  It would be interesting to see if class number formulas can also be obtained from the present log-algebraic methods.
\end{remark}

\section{Characteristic polynomials and Goss $L$-series} \label{S:Charpoly}

We continue with our Drinfeld module $\phi: A \to A[\tau]$ from \S\ref{S:Notation}, and we fix an irreducible polynomial $f \in A_+$ of degree $d$.  We assume that the reduction $\ophi : A \to \FF_f[\tau]$ has rank $r_0$ as in \eqref{E:ophi}.  We review some results due to Gekeler~\cite{Gekeler83}, \cite{Gekeler91}.

Assume for the moment that $r_0 \geq 1$.  Let $\lambda \in A_+$ be irreducible with $\lambda \neq f$, and let
\begin{equation} \label{E:charpoly}
  P_f(x) = x^{r_0} + a_{r_0-1} x^{r_0-1} + \cdots + a_0 \in A[x]
\end{equation}
be the characteristic polynomial of the $q^d$-th power Frobenius $\tau^d$ acting on the Tate module $T_{\lambda}(\ophi)$.  We have the following formulation of a result of Gekeler and subsequent corollary.

\begin{theorem}[{Gekeler~\cite[Thm.~5.1]{Gekeler91}}] \label{T:Gekeler}
Suppose that the rank of $\ophi$ is $r_0 \geq 1$.  For the characteristic polynomial $P_f(x) \in A[x]$ in \eqref{E:charpoly}, the following hold.
\begin{enumerate}
\item[(a)] For some $c_f \in \FF_q^{\times}$, we have $a_0 = c_f^{-1} f$.
\item[(b)] The ideal $(P_f(1)) \subseteq A$ is an Euler-Poincar\'{e} characteristic for $\ophi(\FF_f)$.
\item[(c)] The roots $x_1, \dots, x_{r_0}$ of $P_f(x)$ in $\oK$ satisfy $\deg_{\theta} x_i \leq d/r_0$.
\end{enumerate}
\end{theorem}

\begin{corollary} \label{C:Pf1}
Continuing with the hypotheses of Theorem~\ref{T:Gekeler}, the following  hold.
\begin{enumerate}
\item[(a)] For $1 \leq \ell \leq r_0$, we have $\deg_\theta a_{r_0-\ell} \leq \ell d/r_0$.
\item[(b)] $|\ophi(\FF_f)| = c_f P_f(1)$.
\end{enumerate}
\end{corollary}

\begin{proof}
We see that $a_{r_0-\ell}$ is $(-1)^{\ell}$ times the degree $\ell$ elementary symmetric polynomial in $x_1, \dots, x_{r_0}$, and so (a) follows from Theorem~\ref{T:Gekeler}(c).  Now $\deg a_0 = d$ by Theorem~\ref{T:Gekeler}(a), which is strictly larger than $\ell d/r_0$ for $1 \leq \ell \leq r_0-1$, and so
\[
  P_f(1) = 1 + a_{r_0-1} + \cdots + a_1 + c_f^{-1}f \in A
\]
has leading coefficient (with respect to $\theta$) equal to $c_f^{-1}$.  Therefore (b) follows.
\end{proof}

\begin{remark}
Determining the exact value of $c_f$ in Theorem~\ref{T:Gekeler}(a) can be intricate, but work of Hsia and Yu~\cite[Thm.~3.1]{HsiaYu00} evaluates it completely.  Knowing its precise value will not be necessary for our present considerations, but on the other hand it is necessary for calculation (e.g., see \eqref{E:mufHasse}).
\end{remark}

We now define the Goss $L$-series $L(\phi,s)$ and $L(\phi^{\vee},s)$ for $\phi$ over $K$, as found in \cite{Goss83}, \cite{Goss92}, \cite[Ch.~8]{Goss}.  If $r_0 \geq 1$, we let
\[
  Q_f(x) = 1 + a_{r_0-1} x + \dots + a_0 x^{r_0} = x^{r_0} P_f(1/x)
\]
be the reciprocal polynomial of $P_f(x)$.  If $r_0 = 0$, then we set $P_f(x) = Q_f(x) = 1$.  Then as we vary over all $f \in A_+$ irreducible, the Goss $L$-function for $\phi$ over $K$ is defined by the Euler product
\begin{equation} \label{E:Lphi}
  L(\phi,s) = \prod_f Q_f \bigl( f^{-s} \bigr)^{-1} = \prod_{f} \frac{1}{1 + a_{r_0-1} f^{-s} + \cdots + c_f^{-1} f^{1-r_0 s}}, \quad s \in \ZZ.
\end{equation}
For $f$ of good reduction ($f \nmid \kappa_r$), we have $r_0 = r$, and so $L(\phi,s)$ is a Goss $L$-series of degree~$r$.  We postpone considerations of convergence in $K_\infty$ until Corollary~\ref{C:convergence}.

\begin{remark}
If $\phi$ has good reduction at $f$, then $P_f(x)$ is the same as the characteristic polynomial of geometric Frobenius $\sigma_f \in \Gal(K^{\sep}/K)$ acting on the $\lambda$-adic Tate module $T_{\lambda}(\phi)$ of $\phi$ itself (e.g., see Goss~\cite[\S 3.4]{Goss92}).  For primes of bad reduction defining the Euler factors of global $L$-series is more subtle, but Gardeyn~\cite[Thm.~7.3]{Gardeyn02} has shown that the characteristic polynomial obtained through the Galois action on $T_{\lambda}(\phi)$ is an element of $A[x]$ and independent of $\lambda$ for all but finitely many $\lambda$.  Moreover he worked out global $L$-factors at bad primes for Drinfeld modules~\cite[\S 8.4]{Gardeyn02}.  However, the polynomials defining Gardeyn's Euler factors may not agree with $P_f(x)$ when $\phi$ has bad reduction at $f$.  On the other hand, we will continue with the definition of $P_f(x)$ for all $f$, as it dovetails with Taelman's $L$-values in Theorem~\ref{T:Taelman} and the log-algebraic identities we prove in Theorem~\ref{T:Main}.
\end{remark}

Although $L(\phi,s)$ is natural to define, it is in fact the dual $L$-series $L(\phi^{\vee},s)$ that appears in Taelman's formulas \cite{Taelman09}, \cite{Taelman10}, \cite{Taelman12}.  Returning to fixed $f$, if we consider instead the Frobenius action on the dual of $T_{\lambda}(\ophi)$, we have characteristic polynomials of $\tau^d$ given by
\begin{align*}
  P_f^{\vee}(x) &= x^{r_0} + \frac{c_f a_1}{f} x^{r_0-1} + \cdots + \frac{c_f}{f}, \\
  Q_f^{\vee}(x) &= 1 + \frac{c_f a_1}{f} x + \cdots + \frac{c_f a_{r_0-1}}{f} x^{r_0-1} + \frac{c_f}{f} x^{r_0}= \frac{P_{f}(x)}{P_{f}(0)}.
\end{align*}
We set the global $L$-series
\begin{equation}\label{E:Lphidual}
  L(\phi^{\vee},s) = \prod_f Q_f^{\vee} \bigl( f^{-s} \bigr)^{-1}, \quad s \in \ZZ.
\end{equation}
Now using Corollary~\ref{C:Pf1}(b), we find that
\begin{equation} \label{E:PfQf1}
   \frac{|\FF_f|}{|\ophi(\FF_f)|}=\frac{f}{c_fP_f(1)}=\frac{f}{f + c_f a_1 + \dots + c_f a_{r_0-1} + c_f}=P_f^{\vee}(1)^{-1} = Q_f^{\vee}(1)^{-1},
\end{equation}
and recalling \eqref{E:TaelmanL} we see (as observed by Taelman~\cite[Rem.~5]{Taelman12}) that~\eqref{E:PfQf1} implies
\begin{equation}
  L(\phi/A) = L(\phi^{\vee},0).
\end{equation}
We now define a function $\mu: A_{+}\to A$ by requiring it to satisfy the Dirichlet series expansion
\begin{equation} \label{E:mudef}
  L(\phi^{\vee},s-1) = \sum_{a \in A_+} \frac{\mu(a)}{a^s}.
\end{equation}
Notice that if we let
\[
  R_f(x):=Q_{f}^{\vee}(fx) = 1 - b_1 x - b_2 f x^2 - \cdots - b_{r-1} f^{r-2}x^{r-1} - b_{r} f^{r-1} x^{r},
\]
with
\begin{equation} \label{E:belldef}
b_{\ell} = \begin{cases}
-c_f a_{\ell} & \textup{if $1 \leq \ell \leq r_0-1$}, \\
-c_f & \textup{if $\ell = r_0$}, \\
0 & \textup{if $r_0 < \ell \leq r$,}
\end{cases}
\end{equation}
then
\begin{equation} \label{E:LphidualEuler}
  L(\phi^{\vee},s-1) = \prod_{f} Q_f^{\vee} \bigl( f^{-(s-1)} \bigr)^{-1} = \prod_{f} R_f \bigl( f^{-s} \bigr)^{-1}.
\end{equation}
It follows that in order for \eqref{E:mudef} to hold for $f\in A_+$ irreducible, we must have for  $r_0 \geq 1$
\begin{equation}
  \mu(f) = b_1 = -c_f a_1,
\end{equation}
and for $r_0 = 0$,
\begin{equation}
  \mu(f^m) := 0, \quad \forall\,m \geq 1.
\end{equation}
Furthermore, the expansion \eqref{E:mudef} implies that $\mu$ is multiplicative and provides a recursion at powers of irreducible polynomials. We collect the properties of $\mu$ in the following lemma, where will use the convention that $\mu(b) = 0$ for $b \in K \setminus A_+$.

\begin{lemma} \label{L:muprops}
The function $\mu : A_+ \to A$ satisfies the following properties.  Fix $f \in A_+$ irreducible of degree $d$ such that $\ophi$ has rank $r_0 \leq r$.
\begin{enumerate}
\item[(a)] $\mu$ is multiplicative.
\item[(b)] For any integer $m > -r$,
\[
  \mu\bigl(f^{m+r} \bigr) = \mu(f) \mu \bigl(f^{m+r-1}\bigr) +
  \sum_{\ell=2}^{r} b_{\ell} f^{\ell-1} \mu\bigl( f^{m+r-\ell} \bigr).
\]
\item[(c)] For any $a \in A_+$,
\[
  \mu(fa) = \mu(f) \mu(a) + \sum_{\ell=2}^{r} b_{\ell} f^{\ell-1} \mu \biggl( \frac{a}{f^{\ell-1}} \biggr).
\]
\item[(d)] For any $m \geq 1$,
\[
  \deg \mu \bigl(f^m \bigr) \leq  \biggl( 1 - \frac{1}{r_0} \biggr) md.
\]
Moreover for any $a \in A_+$,
\[
  \deg \mu(a) \leq \biggl( 1 - \frac{1}{r} \biggr) \deg a.
\]
\end{enumerate}
\end{lemma}

\begin{proof}
Part (a) follows from the definition of $\mu$ via the Euler product expansion  in \eqref{E:LphidualEuler}.  Likewise the Euler product implies that we have a generating series
\[
  \sum_{m=0}^{\infty} \mu(f^m) x^m = R_f(x)^{-1} =
  \frac{1}{1 - b_1 x - b_2 f x^2 - \cdots - b_{r} f^{r-1} x^{r}},
\]
and this produces the recursion in (b).  Part (c) then follows by writing $a = b f^{m+r-1}$ with $\gcd(a,b) = 1$ and then using the recursion from (b).  For (d) we first consider the case that $a$ is a power of $f$.  We note from Corollary~\ref{C:Pf1}(a) and \eqref{E:belldef} that
\[
  \deg \bigl( b_{\ell} f^{\ell-1} \bigr) \leq \frac{(r_0 - \ell)d}{r_0} + (\ell-1)d =
  \ell d \biggl( 1 - \frac{1}{r_0} \biggr),
\]
and since $\mu(f) = b_1$, we have right away that
\[
  \deg \mu(f) \leq d \biggl( 1 - \frac{1}{r_0} \biggr)
\]
as desired.  For higher powers of $f$ we see from (b) that for $m \in \ZZ$,
\begin{multline*}
  \deg \mu \bigl( f^{m+r_0} \bigr) \leq \max \bigl\{ \deg b_\ell + (\ell - 1) d + \deg \mu\bigl( f^{m + r_0 - \ell} \bigr) \bigm| 2 \leq \ell \leq r_0 \bigr\} \\
  \cup \bigl\{ \deg \mu(f) + \deg \mu\bigl( f^{m+r_0-1} \bigr) \bigr\},
\end{multline*}
and so by induction we find that (d) is satisfied for powers of $f$.  The final result then follows by the multiplicativity of $\mu$, using also the fact that $r_0 \leq r$ for all irreducible $f$.
\end{proof}

\begin{corollary} \label{C:convergence}
As functions of $s \in \ZZ$,
\begin{enumerate}
\item[(a)] $L(\phi^{\vee},s)$ converges in $K_\infty$ for $s \geq 0$;
\item[(b)] $L(\phi,s)$ converges in $K_\infty$ for $s \geq 1$.
\end{enumerate}
\end{corollary}

\begin{proof}
Part (a) is a consequence of Lemma~\ref{L:muprops}(d), and part (b) follows similarly using Corollary~\ref{C:Pf1}(a).
\end{proof}

\section{Log-algebraicity and special $L$-values} \label{S:LogAlg}

In this section we state the main theorem of the paper, which is a version of Anderson's log-algebraicity theorem for general Drinfeld $A$-modules defined over $A$.  We first recall Anderson's $\star$-operation \cite[\S 3]{And96} for the Carlitz module.

We let $C$ denote the Carlitz $A$-module, which is defined by
\begin{equation}\label{E:DefCarlitz}
C_{\theta} = \theta + \tau.
\end{equation}
For a variable $x$, we define
\[
  \star : A \times K[x] \to K[x]
\]
by
\begin{equation} \label{E:stardef}
  (a \star \beta)(x) := \beta(C_a(x)), \quad a \in A,\ \beta \in K[x].
\end{equation}
One checks the following properties~\cite[\S 3]{And96} for $a$, $b \in A$, $\ell \in K$, and $\beta$, $\gamma \in K[x]$:
\begin{align} \label{E:starprops1}
  a \star (b \star \beta) &= (ab) \star \beta, &
  a \star (\beta + \gamma) &= a \star \beta + a \star \gamma, \\
\label{E:starprops2}
  a \star (\ell \beta) &= \ell(a \star \beta), &
  a \star (\beta\gamma) &= (a \star \beta)(a \star \gamma).
\end{align}
It is notable that in general $(a+b) \star \beta \neq a \star \beta + b \star \beta$, so the $\star$-operation is only a monoidal operation of $A$ and not a ring operation.  Also evident is that the $\star$-operation preserves $A[x]$.

We fix the Drinfeld module $\phi: A \to A[\tau]$ from \eqref{E:phidef} and the associated multiplicative function $\mu : A_+ \to A$ from \eqref{E:mudef}.  For $\beta \in K[x]$, we set
\begin{equation} \label{E:Lphibeta}
  \cL_{\phi}(\beta,z) := \sum_{a \in A_+} \frac{\mu(a)\, (a \star \beta)(x)}{a} \, z^{q^{\deg a}}
   \in \power{K[x]}{z},
\end{equation}
and then our main result is the following.

\begin{theorem} \label{T:Main}
For any $\beta \in A[x]$, the power series
\[
  \cE_{\phi}(\beta,z) := \exp_{\phi} \bigl( \cL_{\phi}(\beta,z) \bigr) \in \power{K[x]}{z},
\]
is in fact in $A[x,z]$.
\end{theorem}

The proof of Theorem~\ref{T:Main} is inspired by the method of Anderson in~\cite{And96}, and it takes up \S\ref{S:Integrality}--\S\ref{S:Degrees}.  The case when $\phi$ is itself the Carlitz module $C$ is a special case of Anderson's theorem~\cite[Thm.~3]{And96}.  In the general case the coefficients $\mu(a)$ impose complications, which will require several intermediary results.  For bounds on the degree of $\cE_{\phi}(\beta,z)$ in $z$, see~\eqref{E:degree}.

\begin{remark}
It would be interesting to explore a proof along the lines of Thakur~\cite[\S 8.10]{Thakur} via explicit formulas for power sums, though this appears to be difficult because of the varying nature $\mu(a)$ (cf.\ Lemma~\ref{L:S1beta}).
\end{remark}

\begin{remark}
It is a natural question to ask whether one can work with a similar $\star$-operation defined using another Drinfeld $A$-module instead of the Carlitz module for the series $\cL_{\phi}(\beta,z)$, including one defined using $\phi$ itself instead of $C$.  At present the authors do not know how to generalize the approach in this paper to tackle this issue and do not know what to predict as far as connections with related arithmetic points of view.  However, the $\star$-operation given by the Carlitz module is the best suited for the study of $L$-values twisted by Dirichlet characters.
\end{remark}

For the remainder of this section we demonstrate how specializations of $\cE_{\phi}(\beta,z)$ can be used to express special $L$-values in terms of $\oK$-linear combinations of Drinfeld logarithms at algebraic points.  We start with constructions similar to Anderson~\cite[\S 4]{And96}, but see also~\cite[\S 3]{LP13}.  Fixing an irreducible element $\wp \in A_+$ with $\deg \wp = d$, we let $\chi : A \to \overline{\FF}_q$ be a Dirichlet character modulo $\wp$, and we investigate twisted $L$-series,
\[
  L(\phi^{\vee},\chi,s-1) = \sum_{a \in A_+} \frac{\chi(a)\mu(a)}{a^s},
\]
especially their values at $s=1$ (similar to Corollary~\ref{C:convergence}, this series converges for $s \geq 1$).  We set $\be(z) = \exp_{C}(\tpi z)$, where $\tpi$ is the period of Carlitz module (see \cite[\S 3.2]{Goss}, \cite[\S 2.5]{Thakur}), and for $m \geq 1$ we let
\[
  \bu_m(z) = \sum_{a \in A_+} \frac{\mu(a) \be(az)^m}{a} \in \power{\CC_\infty}{z}.
\]
The connection with $L$-series arises from Anderson's dual coefficients~\cite[Prop.~10]{And96}.  For $a \in A$ with $\wp \nmid a$ and for $1 \leq m \leq q^d-1$, Anderson defines $\be_m^*(a) \in A[\be(1/\wp)]$ such that for any $a$, $b$ relatively prime to $\wp$,
\[
  \sum_{m=1}^{q^d-1} \be_m^*(a) \be( b/\wp)^m =
  \begin{cases}
  \wp & \textup{if $a \equiv b \pmod{\wp}$,} \\
  0 & \textup{otherwise.}
  \end{cases}
\]
{From} this orthogonality identity we find that
\begin{equation}
  \sum_{\substack{n \in A_+ \\ bn \equiv a \smod{\wp}}} \frac{\mu(n)}{n}
  = \frac{1}{\wp} \sum_{n \in A_+} \frac{\mu(n)}{n}
  \sum_{m=1}^{q^d-1} \be_m^*(a) \be (bn/\wp)^m
  = \frac{1}{\wp} \sum_{m=1}^{q^d-1} \be_m^{*}(a) \bu_m ( b/\wp).
\end{equation}
Multiplying through by $\chi(a)$ and summing over $a\in \FF_{\wp}^{\times}$ and taking $b=1$, we see that
\begin{equation}
L(\phi^{\vee},\chi,0) = \sum_{m=1}^{q^d-1} \Biggl( \frac{1}{\wp} \sum_{a \in \FF_{\wp}^{\times}} \chi(a) \be_m^{*}(a) \Biggr) \bu_m(1/\wp).
\end{equation}
If we let $\gamma_m = 1/\wp \sum_a \chi(a) \be_m^{*}(a)$, which Anderson terms a root number for $\wp$, then $\gamma_m \in \FF_{\wp}(\theta,\be(1/\wp)) \subseteq \oK$, where $\FF_{\wp}$ is naturally viewed in $\oK \subseteq \CC_{\infty}$.  Furthermore,
\[
  \bu_m(1/\wp) = \sum_{a \in A_+} \frac{\mu(a)\, (a \star x^m)(\be(1/\wp))}{a} = \cL_{\phi}(x^m,1)\big|_{x=\be(1/\wp)},
\]
which converges in $\CC_\infty$ since $\be(a/\wp)$, for $a \in A_+$, ranges over a finite set.  We thus see from Theorem~\ref{T:Main} that
\[
  \exp_{\phi} \bigl( \bu_m(1/\wp) \bigr) \in A[\be(1/\wp)],
\]
and we obtain the following result.

\begin{corollary}\label{C:L-log}
Let $\wp \in A_+$ be irreducible of degree $d$.  Let $\chi : A \to \overline{\FF}_{q}$ be a Dirichlet character modulo $\wp$.  Then for some $s$, $1 \leq s \leq q^d-1$, there are elements $u_1, \dots, u_s \in \CC_\infty$, with $\exp_\phi(u_i) \in A[\be(1/\wp)] \subseteq \oK$, and $\gamma_1, \dots, \gamma_s \in \FF_{\wp}\left( \theta, \be(1/\wp) \right)$ so that
\[
  L(\phi^{\vee},\chi,0) = \sum_{m=1}^s \gamma_m u_m.
\]
\end{corollary}

\begin{remark}
Parts of this corollary are also covered by work of Fang~\cite[Thm.~1.11]{Fang15} on special $L$-values for abelian $t$-modules.
\end{remark}

\begin{corollary}\label{C:Trans-L}
Let $\phi$ be a Drinfeld $A$-module of rank $r$ defined over $A$, and let $\wp \in A_+$ be irreducible. Then for any Dirichlet character $\chi : A \to \overline{\FF}_{q}$ modulo $\wp$, the special $L$-value $L(\phi^{\vee},\chi,0)$ is transcendental over $K$.
\end{corollary}

\begin{proof}
By Lemma~\ref{L:muprops} we see that for any $a\in A_+$ of positive degree,
\[
\left\lvert \frac{\chi(a)\mu(a)}{a}\right\rvert_{\infty}\leq q^{-1/r}< 1 .
\]
Hence, $L(\phi^{\vee},\chi,0)$ is non-vanishing as $|\chi(1)\mu(1)|_{\infty}=1$. The desired result follows by combining Corollary~\ref{C:L-log} and \cite[Thm.~1.1.1]{CP11}, \cite[Thm.~1.1.1]{CP12}.
\end{proof}

\begin{remark}
For a fixed prime $\wp$ of $A$, it is an interesting question to study the transcendence degree of $L(\phi^{\vee},\chi,0)$ with varying Dirichlet characters $\chi$ modulo $\wp$. Using the recent advance on the transcendence theory for Drinfeld modules~\cite{CP11}, \cite{CP12}, this question is related to the computation of the rank of the $A$-submodule of $\phi\bigl(A[\be(1/\wp)]\bigr)$ generated by the special points
\[
\exp_{\phi} \bigl( \bu_m(1/\wp) \bigr),
\]
and one needs additional study to tackle this question (cf.~\cite{LP13} for the rank one case).

\end{remark}

\section{Integrality results} \label{S:Integrality}

This section and the next contain the proof of Theorem~\ref{T:Main}.  We fix $\beta \in A[x]$, and as in the case of Anderson~\cite[Thm.~3]{And96}, we first show that the power series $\cE_\phi(\beta,z)$ has coefficients in $A[x]$.  Then in \S\ref{S:Degrees} we use $\infty$-adic estimates to show that $\cE_\phi(\beta,z)$ is a polynomial in $z$.

Fixing an irreducible polynomial $f \in A_+$ of degree $d$, we let $A_{(f)} = \{ g \in K \mid \ord_f(g) \geq 0 \}$ be the local subring of $K$ of $f$-integral elements.

\begin{theorem} \label{T:Integrality}
Let $\beta \in A[x]$ and let $f \in A_+$ be irreducible.  Then $\cE_{\phi}(\beta,z) \in \power{A_{(f)}[x]}{z}$.
\end{theorem}

The proof of this theorem occupies the rest of this section.  However, as a direct result we see that
\begin{equation} \label{E:EphiIntegral}
  \cE_{\phi}(\beta,z) \in \bigcap_f \power{A_{(f)}[x]}{z} = \power{A[x]}{z}.
\end{equation}
We establish some notation.  For $\beta \in K[x]$, we set
\begin{equation} \label{E:Sidef}
  S_i(\beta) = \sum_{a \in A_{i+}} \frac{\mu(a)\, (a \star \beta)(x)}{a} \in K[x]
\end{equation}
so that
\[
  \cL_{\phi}(\beta,z) = \sum_{i =0}^\infty S_i(\beta) z^{q^i}.
\]
We define $E_i(\beta) \in K[x]$ so that
\begin{equation} \label{E:Eidef}
  \cE_{\phi}(\beta,z) = \exp_{\phi}\bigl( \cL_{\phi}(\beta,z) \bigr)  = \sum_{i=0}^{\infty} E_i(\beta) z^{q^i},
\end{equation}
and thus
\begin{equation} \label{E:EiSum}
  E_i(\beta) = \sum_{j=0}^i \alpha_j S_{i-j}(\beta)^{q^j}.
\end{equation}
By convention, for $i <0$ we set $S_i(\beta) = E_i(\beta)=0$.  We further define
\begin{equation} \label{E:Sistardef}
  S_i^{*}(\beta) := \sum_{\substack{a \in A_{i+} \\ f \nmid a}} \frac{\mu(a)\, (a \star \beta)(x)}{a}.
\end{equation}

\begin{lemma} \label{L:fSistar}
Let $\beta \in K[x]$, and let $f \in A_+$ be irreducible of degree $d$.  Then
\[
  f S_i^{*}(\beta) = f S_i(\beta) -  f \star (\mu(f) S_{i-d}(\beta)) - \sum_{\ell=2}^r f^\ell \star (b_{\ell} S_{i-\ell d}(\beta)).
\]
\end{lemma}

\begin{proof}
Throughout we make frequent use of~\eqref{E:starprops1}--\eqref{E:starprops2}.  We see from~\eqref{E:Sistardef} that
\[
  f S_i^{*}(\beta) = f \Biggl( S_i(\beta) - \sum_{\substack{a \in A_{i+} \\ f \mid a}} \frac{\mu(a)(a \star \beta)}{a} \Biggr)
  = f \Biggl( S_i(\beta) - \sum_{a \in A_{(i-d)+}} \frac{ \mu(fa)((fa)\star \beta)}{fa} \Biggr),
\]
and so by Lemma~\ref{L:muprops}(c),
\begin{align*}
  f S_i^{*}(\beta) &= f S_i(\beta) - \sum_{a \in A_{(i-d)+}} \sum_{\ell = 1}^r \frac{b_{\ell} \mu \bigl(a/f^{\ell-1} \bigr)}{a / f^{\ell-1}} \, ((f a) \star \beta) \\
  &= f S_i(\beta) - \sum_{\ell = 1}^r \sum_{a \in A_{(i-\ell d)+}} \frac{b_\ell \mu(a)}{a} ((f^\ell a) \star \beta) \\
  &= f S_i(\beta) - \sum_{\ell=1}^r f^{\ell} \star \sum_{a \in A_{(i-\ell d)+}} \frac{b_{\ell} \mu(a)}{a} (a \star \beta),
\end{align*}
where in the second equality we have used the property that $\mu(a/f^{\ell-1}) = 0$ if $f^{\ell-1} \nmid a$.  The proof is complete after we recall that $b_1 = \mu(f)$.
\end{proof}

\begin{lemma} \label{L:Eistar}
Let $\beta \in K[x]$, and let $f \in A_+$ be irreducible of degree $d$.  For $i \geq 0$, let
\[
  E_i^*(\beta) = \sum_{j=0}^i \alpha_j f^{q^j} S_{i-j}^{*}(\beta)^{q^j}.
\]
Then
\[
  E_i^*(\beta) = \sum_{k=0}^{rd} \brac{f}{k} E_{i-k}(\beta)^{q^k} - \sum_{\ell=1}^r \sum_{k=\ell d}^{rd} \brac{b_{\ell}}{k - \ell d} \bigl( f^{\ell} \star E_{i-k}(\beta)^{q^{k-\ell d}}\bigr).
\]
\end{lemma}

\begin{proof}
By the functional equation for $\exp_\phi(z)$ in~\eqref{E:expfneq}, we know that
\[
  \phi_f \bigl(\cE_{\phi}(\beta,z) \bigr) = \exp_{\phi} \bigl( f \cL_{\phi}(\beta,z) \bigr).
\]
Using~\eqref{E:brac} the left side of this equation has the expansion
\[
  \phi_f \bigl(\cE_{\phi}(\beta,z) \bigr) = \sum_{i=0}^\infty \sum_{k=0}^{rd} \brac{f}{k} E_i(\beta)^{q^k} z^{q^{i+k}} = \sum_{i=0}^{\infty} \Biggl( \sum_{k=0}^{rd} \brac{f}{k} E_{i-k}(\beta)^{q^{k}} \Biggr) z^{q^i}.
\]
On the other hand for $\exp_{\phi}(f \cL_{\phi}(\beta,z))$ we similarly find the expansion
\[
  \exp_{\phi}\bigl(f \cL_{\phi}(\beta,z) \bigr) = \sum_{i=0}^{\infty} \Biggl( \sum_{j=0}^i \alpha_j f^{q^j} S_{i-j}(\beta)^{q^j} \Biggr) z^{q^i}.
\]
Comparing the coefficients of $z^{q^i}$ we arrive at the identity for $i \geq 0$,
\begin{equation} \label{E:Esum1}
  \sum_{k=0}^{rd} \brac{f}{k} E_{i-k}(\beta)^{q^k} = \sum_{j=0}^i \alpha_j f^{q^j} S_{i-j}(\beta)^{q^j}.
\end{equation}
Now since Corollary~\ref{C:Pf1}(a) and \eqref{E:belldef} imply that $\deg_{\theta} b_{\ell} \leq (1 - \ell/r)d$, it follows that $\deg_{\tau} \phi_{b_{\ell}} \leq (r-\ell)d$, and using a similar calculation to the above we have for $1 \leq \ell \leq r$,
\begin{align*}
\phi_{b_{\ell}}\bigl( \cE_{\phi}(f^{\ell} \star \beta,z^{q^{\ell d}} ) \bigr)
  &= \sum_{i=0}^\infty \Biggl( f^{\ell} \star \sum_{k=\ell d}^{rd} \brac{b_{\ell}}{k-\ell d} E_{i-k}(\beta)^{q^{k-\ell d}} \Biggr) z^{q^i}, \\
\phi_{b_{\ell}}\bigl( \cE_{\phi}(f^{\ell} \star \beta,z^{q^{\ell d}} ) \bigr) &= \exp_{\phi}\bigl( b_{\ell} \cL_{\phi}(f^{\ell}\star \beta,z^{q^{\ell d}}) \bigr) = \sum_{i=0}^\infty \Biggl( f^{\ell} \star \sum_{j=0}^i \alpha_j b_{\ell}^{q^j} S_{i-\ell d - j}(\beta)^{q^j} \Biggr) z^{q^i}.
\end{align*}
Thus for $1 \leq \ell \leq r$ and $i \geq 0$,
\begin{equation} \label{E:Esum2}
  f^{\ell} \star \sum_{k=\ell d}^{rd} \brac{b_{\ell}}{k-\ell d} E_{i-k}(\beta)^{q^{k - \ell d}}
  = f^{\ell} \star \sum_{j=0}^i \alpha_j b_\ell^{q^j} S_{i-\ell d - j}(\beta)^{q^j}.
\end{equation}
Using \eqref{E:Esum1}, \eqref{E:Esum2}, Lemma~\ref{L:fSistar}, and Lemma~\ref{L:Eistar}, we have
\begin{align*}
\sum_{k=0}^{rd} \brac{f}{k} E_{i-k}(\beta)^{q^k} -{} &\sum_{\ell=1}^r \sum_{k=\ell d}^{rd} \brac{b_{\ell}}{k-\ell d} \bigl( f^{\ell} \star E_{i-k}(\beta)^{q^{k - \ell d}} \bigr) \\
  &= \sum_{j=0}^i \alpha_j \biggl( f S_{i-j}(\beta) - \sum_{\ell=1}^r b_{\ell} \bigl( f^{\ell} \star S_{i-\ell d - j}(\beta) \bigr) \biggr)^{q^j}\\
  &= \sum_{j=0}^i \alpha_j f^{q^j} S_{i-j}^* (\beta)^{q^j}\\
  &= E_i^*(\beta),
\end{align*}
and we are done.
\end{proof}

For polynomials $P$, $Q \in A_{(f)}[x]$, we will say that $P \equiv Q \pmod{f}$ if $P-Q \in f A_{(f)}[x]$.

\begin{lemma} \label{L:braccongs}
Let $f \in A_+$ be irreducible of degree $d$.
\begin{enumerate}
\item[(a)] For $0 \leq k \leq d-1$,
\[
  \brac{f}{k} \equiv 0 \pmod{f}.
\]
\item[(b)] For $0 \leq k \leq rd$,
\[
  \brac{f}{k} - \sum_{\ell=1}^{\lfloor k/d \rfloor} \brac{b_{\ell}}{k-\ell d} \equiv 0 \pmod{f}.
\]
\item[(c)] Let $P(x) \in A_{(f)}[x]$.  For $0 \leq k \leq rd$,
\[
  \brac{f}{k} P(x)^{q^k} - \sum_{\ell=1}^{\lfloor k/d \rfloor} \brac{b_{\ell}}{k - \ell d} \bigl(f^{\ell} \star P(x)^{q^{k-\ell d}} \bigr) \equiv 0 \pmod{f}.
\]
\end{enumerate}
\end{lemma}

\begin{proof}
Part (a) is a special case of (b), and is in fact well-known (e.g., see \cite[(1.4)(ii)]{Gekeler91}).  Part (b) is itself a special case of (c), using $P = 1$, but we require (b) first to prove (c).  Let $r_0 \leq r$ be the rank of the reduction $\ophi$ modulo $f$.  If $r_0 = 0$, then this implies that $\brac{f}{k} \equiv 0 \pmod{f}$ for all $0 \leq k \leq rd$, and moreover $b_\ell = 0$ for all $1 \leq \ell \leq r$, so each of (a)--(c) is trivially satisfied.

Assume $r_0 \geq 1$.  Fix an irreducible $\lambda \in A_+$ with $\lambda \neq f$.  We know by Gekeler~\cite[p.~190]{Gekeler91} that the natural map
\[
  \End(\ophi) \otimes_A A_{\lambda} \to \End_{A_{\lambda}} \bigl(T_{\lambda}(\ophi)\bigr)
\]
is injective.  By the Cayley-Hamilton theorem, $P_f(\tau^d) = 0$ in $\End_{A_{\lambda}} (T_{\lambda}(\ophi))$, and so in $\End(\ophi) \otimes_A A_{\lambda}$,
\[
  \tau^{r_0 d} + \ophi_{a_{r_0-1}} \tau^{(r_0-1)d} + \dots + \ophi_{a_0} = 0.
\]
However, the left-hand side is an element of $\End_{\FF_q}(\Ga/\FF_f) \cong \FF_f[\tau]$, and so multiplying by $\ophi_{c_f}$, we have
\[
  \ophi_f - \ophi_{b_1} \tau^d - \cdots - \ophi_{b_{r_0}} \tau^{r_0 d} = \ophi_f - \ophi_{b_1} \tau^d - \cdots - \ophi_{b_r} \tau^{rd} = 0,
\]
which is an equality in $\FF_f[\tau]$.  Thus in $A[\tau]$,
\[
  \sum_{k=0}^{rd} \brac{f}{k} \tau^k - \sum_{\ell = 1}^r \Biggl( \sum_{k=0}^{(r-\ell)d} \brac{b_{\ell}}{k} \tau^k \Biggr) \tau^{\ell d} \equiv 0 \pmod{f},
\]
and by reordering the sum, we see that
\[
  \sum_{k=0}^{rd} \Biggl( \brac{f}{k} - \sum_{\ell=1}^{\lfloor k/d \rfloor} \brac{b_{\ell}}{k- \ell d} \Biggr) \tau^k \equiv 0 \pmod{f}.
\]
This can only hold if each coefficient is $0$ modulo $f$, which proves (b).

For (c) we note that (b) implies for $a \in A_{(f)}$ and $0 \leq k \leq rd$ that
\begin{equation} \label{E:bracacong}
  \brac{f}{k} a^{q^k} - \sum_{\ell=1}^{\lfloor k/d \rfloor} \brac{b_{\ell}}{k-\ell d} a^{q^{k-\ell d}} \equiv 0 \pmod{f},
\end{equation}
since $a^{q^k} \equiv a^{q^{k-\ell d}} \pmod f$ for all $\ell$.  By \eqref{E:starprops1}--\eqref{E:starprops2}, it suffices to prove (c) for $P = a x^i$.  We recall (see \cite[\S 3.6]{Goss}) that for $\ell \geq 1$,
\[
  C_{f^{\ell}}(x) \equiv x^{q^{\ell d}} \pmod{f},
\]
and so
\[
f^{\ell} \star P(x) \equiv a x^{i q^{\ell d}} \pmod{f}.
\]
For $0 \leq k \leq rd$, this congruence then implies
\begin{multline*}
  \brac{f}{k} P(x)^{q^k} - \sum_{\ell=1}^{\lfloor k/d \rfloor} \brac{b_{\ell}}{k-\ell d} \bigl( f^{\ell}\star P(x)^{q^{k-\ell d}} \bigr) \\
  \equiv \Biggl( \brac{f}{k} a^{q^{k}} - \sum_{\ell=1}^{\lfloor k/d \rfloor} \brac{b_{\ell}}{k- \ell d} a^{q^{k-\ell d}} \Biggr) x^{iq^k} \pmod{f},
\end{multline*}
and part (c) then follows from \eqref{E:bracacong}.
\end{proof}

\begin{proof}[Proof of Theorem~\ref{T:Integrality}]
We prove $E_i(\beta) \in A_{(f)}[x]$ by induction on $i$.  We note that $E_i(\beta) = 0$ for $i <0$ and $E_0(\beta) =  S_0(\beta) = \beta \in A[x]$, which establish the base cases.  Now assume that $E_{j}(\beta) \in A_{(f)}[x]$ for all $j < i$.
By Lemma~\ref{L:Eistar} (and $\brac{f}{0} = f$), we have
\begin{multline*}
  E_i(\beta) = \frac{1}{f} \Biggl( E_i^{*}(\beta) - \sum_{k=1}^{d-1} \brac{f}{k} E_{i-k}(\beta)^{q^k} \Biggr.
   \\ \Biggl. {}- \sum_{k=d}^{rd} \biggl( \brac{f}{k} E_{i-k}(\beta)^{q^k} - \sum_{\ell=1}^{\lfloor k/d \rfloor} \brac{b_{\ell}}{k- \ell d} \bigl( f^{\ell} \star E_{i-k}(\beta)^{q^{k-\ell d}}\bigr) \biggr) \Biggr).
\end{multline*}
We consider each of the terms on the right-hand side of this equation.  Recall from Lemma~\ref{L:Eistar} that
\[
  \frac{1}{f} E_i^{*}(\beta) = \sum_{j=0}^i \alpha_j f^{q^j-1} S_{i-j}^*(\beta)^{q^j}.
\]
By construction in~\eqref{E:Sistardef}, $S_{i-j}^*(\beta) \in A_{(f)}[x]$.  Also by \cite[Thm.~3.1]{EP13},
\[
  \ord_f(\alpha_j) \geq -\ord_f \bigl( (\theta^{q^j} - \theta)(\theta^{q^j}-\theta^q) \cdots (\theta^{q^j} - \theta^{q^{j-1}} ) \bigr) = - \frac{q^j - q^{j- d\lfloor j/d \rfloor}}{q^d-1}
  \geq -(q^j-1),
\]
where the middle equality follows from the fact that $\ord_f(\theta^{q^k}-\theta) = 1$ if $d \mid k$, and $0$ otherwise (e.g., see~\cite[\S 2.1]{PLogAlg}).  Thus $\frac{1}{f} E_i^{*}(\beta) \in A_{(f)}[x]$ as desired.

By Lemma~\ref{L:braccongs}(a), for $1 \leq k \leq d-1$ we have that $\frac{1}{f}\brac{f}{k} \in A_{(f)}$, and by the induction hypothesis $E_{i-k}(\beta) \in A_{(f)}[x]$.  Thus $\frac{1}{f} \brac{f}{k} E_{i-k}(\beta) \in A_{(f)}[x]$.  For $d \leq k \leq rd$, taking $P = E_{i-k}(\beta)$ in Lemma~\ref{L:braccongs}(c) and using the induction hypothesis provides the $f$-integrality of the remaining sum.
\end{proof}

\section{Degree estimates and completion of the proof} \label{S:Degrees}

After proving Theorem~\ref{T:Integrality} the second part of Anderson's method is to estimate the $\infty$-adic size of the coefficients $E_i(\beta)$ from \eqref{E:Eidef} by defining a norm on $K[x]$ in which $A[x]$ is discrete.  The essential idea is to show that their size goes to $0$ as $i \to \infty$ and then to use their integrality to prove that $E_i(\beta)=0$ for $i \gg 0$.

We recall a definition of Anderson~\cite[\S 3.5]{And96}.  The Carlitz exponential is given by the power series
\begin{equation} \label{E:expC}
  \exp_C(z) = z + \sum_{i=1}^\infty \frac{z^{q^i}}{(\theta^{q^i} - \theta)(\theta^{q^i} - \theta^q) \cdots (\theta^{q^i} - \theta^{q^{i-1}} )},
\end{equation}
and its period lattice $\Lambda_C = \ker \exp_C$ has the form $\Lambda_C = A\tpi$, where $\tpi \in K_{\infty}((-\theta)^{1/(q-1)})$ is called the Carlitz period (see \cite[\S 3.2]{Goss}, \cite[\S 2.5]{Thakur}).  We let $\II = K_\infty \cdot \tpi \subseteq \CC_\infty$ be the imaginary axis, which has the properties that $\II/\Lambda_C$ is compact and that $C_{\tor}=K\cdot \tpi \subseteq \exp_C(\II)$, where $C_{\tor}$ denotes the torsion submodule of $C$.  Anderson then defines for $\beta \in K[x]$,
\begin{equation} \label{E:normdef}
  \norm{\beta} := \sup_{u \in \II}\, \bigl| \beta\bigl(\exp_C(u) \bigr) \bigr|_{\infty}
  = \sup_{a \in A}\sup_{u \in \II}\, \bigl| (a \star \beta)\bigl( \exp_C(u) \bigr) \bigr|_{\infty},
\end{equation}
which is well-defined since by \eqref{E:stardef},
\[
  \sup_{a \in A}\sup_{u \in \II}\, \bigl| (a \star \beta)\bigl( \exp_C(u) \bigr) \bigr|_{\infty}
  = \sup_{a \in A}\sup_{u \in \II}\, \bigl| \beta\bigl( C_a(\exp_C(u)) \bigr) \bigr|_{\infty}
  = \sup_{a \in A}\sup_{u \in \II}\, \bigl| \beta\bigl( \exp_C(au) \bigr) \bigr|_{\infty}.
\]
Anderson proves the following properties of $\norm{\,\cdot\,}$.

\begin{proposition}[{Anderson \cite[Prop.~2]{And96}}] \label{P:normprops}
The function $\norm{\,\cdot\,}$ defines an ultrametric norm on $K[x]$ that is invariant under the $\star$-operation.  In particular, for all $\beta$, $\gamma \in K[x]$, we have $\norm{\beta} < \infty$,
\[
  \norm{\beta+\gamma} \leq \max(\norm{\beta},\norm{\gamma}), \quad
  \norm{\beta\gamma} \leq \norm{\beta}\norm{\gamma}, \quad
  \norm{\beta} = 0\ \Rightarrow\ \beta = 0.
\]
For all $a \in A$, we have $\norm{a \star \beta} = \norm{\beta}$, and if $\xi \in C_{\tor}$, then $|\beta(\xi)|_{\infty} \leq \norm{\beta}$.  Furthermore, the following properties hold.
\begin{enumerate}
\item[(a)] If $\beta \in A[x]$ and $\norm{\beta} \leq 1$, then $\beta \in \FF_q$.
\item[(b)] If $\beta \in A[x]$ and $\norm{\beta} < 1$, then $\beta = 0$.
\end{enumerate}
Thus the ring $A[x]$ is a discrete subspace of $K[x]$ with respect to $\norm{\,\cdot\,}$.
\end{proposition}

Recalling the defining polynomial for $\phi$ in \eqref{E:phidef}, we set
\[
  d_0 := \max ( \deg \kappa_1, \dots, \deg \kappa_r ) \in \ZZ
\]
and note that $d_0 \geq 0$.  For $\beta \in K[x]$, $\beta \neq 0$, we define positive $j_0(\beta) \in \QQ$ by
\[
  |\theta|_{\infty}^{j_0(\beta)} = \max( 1, \norm{\beta} ).
\]
We set $j_0(0) = 0$.  We recall that Anderson~\cite[p.~188]{And96} showed that $j_0(x^m) = m/(q-1)$. Although we will not need evaluations of $j_0$ for more general $\beta$, we remark that Anderson's evaluation together  with the ultrametric properties of $\norm{\,\cdot\,}$ can be used to evaluate $j_0(\beta)$ under various restrictions on the coefficients of $\beta$; for instance $j_0(\beta)=\frac{\deg(\beta)}{q-1}$ for all $\beta\in \FF_q[x]$ (i.e., if $\beta$ has constant coefficients).  Our main result of this section is the following.

\begin{theorem} \label{T:Estimates}
Let $\beta \in K[x]$.  For any $i \geq 0$,
\[
  \norm{E_i(\beta)} \leq |\theta|_{\infty}^{q^i(j_0(\beta) + d_0/(q-1) - i/r)}.
\]
Moreover, if $\beta \in A[x]$ and if
\[
 i > r \biggl( j_0(\beta) + \frac{d_0}{q-1} \biggr),
\]
then $\norm{E_i(\beta)} < 1$ and thus $E_i(\beta) = 0$.
\end{theorem}

\begin{proof}[Proof of Theorem~\ref{T:Estimates}]
Let $\beta \in K[x]$.  We first estimate $\norm{S_i(\beta)}$ for $i \geq 0$ using \eqref{E:Sidef}.  We see that
\begin{equation} \label{E:normSibeta}
  \norm{S_i(\beta)} \leq \max_{a \in A_{i+}} \biggl\{ \biggl| \frac{\mu(a)}{a} \biggr|_{\infty} \cdot \norm{a \star \beta} \biggr\}
  \leq \norm{\beta} \cdot |\theta|_{\infty}^{-\lceil i/r\rceil} \leq \norm{\beta} \cdot |\theta|_{\infty}^{-i/r},
\end{equation}
where the second inequality follows from the $\star$-invariance of $\norm{\,\cdot\,}$ and Lemma~\ref{L:muprops}(d) (noting that $\deg \mu(a)$ must be an integer).  The last inequality is included for ease of use.  Now by~\eqref{E:EiSum},
\begin{equation} \label{E:normEibeta}
  \norm{E_i(\beta)} \leq \max_{0 \leq j \leq i} \Bigl\{ |\alpha_j|_{\infty} \cdot \norm{S_{i-j}(\beta)}^{q^j} \Bigr\}.
\end{equation}
By \cite[Thm.~3.1, Eq.~(28)]{EP13}, we see that
\[
  |\alpha_j|_\infty \leq |\theta|_{\infty}^{q^j (d_0/(q-1)-j/r)},
\]
and so combining these estimates we have
\begin{align*}
  \norm{E_i(\beta)} &\leq \max_{0 \leq j \leq i} \Bigl\{ \norm{\beta}^{q^j} \cdot |\theta|_{\infty}^{q^j(d_0/(q-1) - j/r - (i-j)/r)} \Bigr\} \\
  &\leq \max_{0 \leq j \leq i}\, \Bigl\{ |\theta|_{\infty}^{q^j( j_0(\beta) + d_0/(q-1) - i/r)} \Bigr\} \\
  &= |\theta|_{\infty}^{q^i(j_0(\beta) + d_0/(q-1) - i/r)}.
\end{align*}
Now if $\beta \in A[x]$, then $E_i(\beta) \in A[x]$ by Theorem~\ref{T:Integrality} and \eqref{E:EphiIntegral}.  If $i > r(j_0(\beta) + d_0/(q-1))$, then these estimates imply $\norm{E_i(\beta)} < 1$.  Proposition~\ref{P:normprops}(b) then implies $E_i(\beta) = 0$.
\end{proof}

\begin{proof}[Proof of Theorem~\ref{T:Main}]
For $\beta \in A[x]$, Theorem~\ref{T:Integrality} implies that $\cE_{\phi}(\beta,z) \in \power{A}{z}$, and Theorem~\ref{T:Estimates} implies that the coefficients $E_i(\beta)$ of $\cE_{\phi}(\beta,z)$ are eventually $0$.
\end{proof}

\section{Examples} \label{S:Examples}

Theorem~\ref{T:Estimates} implies that
\begin{equation} \label{E:degree}
\deg_z( \cE_\phi(\beta,z)) \leq q^{ r(j_0(\beta) + d_0/(q-1))},
\end{equation}
although we will see in this section that this bound need not be sharp.  However, we find that for given $\beta \in A[x]$, the polynomial $\cE_\phi(\beta,z) \in A[x,z]$ is effectively computable.  Indeed we can calculate $E_i(\beta)$ for small~$i$ using \eqref{E:EiSum}, and thus we aim to calculate
\[
  S_i(\beta) = \sum_{a \in A_{i+}} \frac{\mu(a) (a\star \beta)(x)}{a}
\]
and the coefficients $\alpha_j$ of $\exp_{\phi}(z)$.  Unfortunately, direct computation of $\mu(a)$ based simply on the defining coefficients of $\phi$ and the factorization of $a$ can be complicated.

We consider cases where $\phi$ has rank $2$, and we assume $\phi_\theta = \theta + g\tau + \Delta \tau^2$, for $g$, $\Delta \in A$, and we take $\beta = x^m$, $m \geq 0$.  Letting $d_0 = \max( \deg g, \deg \Delta)$, we will assume that
\begin{equation}\label{E:md0}
  0 \leq m + d_0 < \tfrac{3}{2}(q-1),
\end{equation}
and since $j_0(x^m) = m/(q-1)$ by~\cite[p.~188]{And96}, we see right away from \eqref{E:degree} that
\begin{equation} \label{E:degE3}
  i > 2 \quad \Rightarrow \quad E_i(x^m) =0.
\end{equation}

\begin{example} \label{Ex:one}
Continuing with the notation above, we assume further that $0 \leq m \leq q-2$ and $0 \leq m+d_0 \leq q-1$.  Then we claim
\begin{equation} \label{E:cEm}
  \cE_\phi(x^m,z) = x^m z.
\end{equation}
Since $E_0(x^m) = x^m$, by \eqref{E:degE3} it remains to show that $E_1(x^m) = E_2(x^m) = 0$.
We first turn to $E_1(x^m) = S_1(x^m) + \alpha_1 S_0(x^m)^q$, and so
\[
\norm{E_1(x^m)} \leq \max\bigl( \norm{S_1(x^m)}, |\alpha_1|_{\infty} \norm{S_0(x^m)}^q \bigr).
\]
We note from~\eqref{E:normSibeta} that $\norm{S_1(x^m)} \leq \norm{x^m} \cdot |\theta|_{\infty}^{-1}$ and from~\cite[Eq.~(24)]{EP13} that $\alpha_1 = g/(\theta^q-\theta)$, and so $\deg(\alpha_1) =\deg g - q$.  By our chosen inequalities,
\[
  \norm{S_1(x^m)} \leq |\theta|_{\infty}^{m/(q-1) - 1} < 1,
\]
and, using that $S_0(x^m) =x^m$,
\[
  |\alpha_1|_{\infty} \norm{x^m}^q \leq |\theta|_{\infty}^{d_0-q} \cdot |\theta|_{\infty}^{mq/(q-1)} = |\theta|_{\infty}^{m+d_0 - (q-1) + m/(q-1)-1} < 1.
\]
Thus $\norm{E_1(x^m)} < 1$, and so $E_1(x^m) =0$ by Proposition~\ref{P:normprops}(b).  Likewise,
\[
  \norm{E_2(x^m)} \leq \max \bigl( \norm{S_2(x^m)}, |\alpha_1|_{\infty} \norm{S_1(x^m)}^q, |\alpha_2|_{\infty} \norm{S_0(x^m)}^{q^2} \bigr).
\]
We see from~\eqref{E:normSibeta} that
\begin{equation} \label{E:normS2}
  \norm{S_2(x^m)} \leq |\theta|_{\infty}^{m/(q-1)} \cdot |\theta|_{\infty}^{-1} < 1.
\end{equation}
Also,
\begin{equation} \label{E:normalpha1S1}
  |\alpha_1|_\infty \norm{S_1(x^m)}^q \leq |\theta|_{\infty}^{d_0-q} \cdot |\theta|_{\infty}^{mq/(q-1) - q} \leq |\theta|_{\infty}^{-1} \cdot |\theta|_{\infty}^{q(m/(q-1)-1)} < 1,
\end{equation}
where the second inequality follows since $d_0 \leq q-1$.  Finally, by \cite[Eq.~(24)]{EP13}, we see that $\deg(\alpha_2) \leq \max ( (q+1)d_0 - 2q^2,d_0-q^2 ) = d_0-q^2$,
and so
\begin{equation} \label{E:normalpha2S0}
  |\alpha_2|_{\infty} \norm{S_0(x^m)}^{q^2} \leq |\theta|_{\infty}^{d_0 -q^2} \cdot |\theta|_{\infty}^{mq^2/(q-1)}
  = |\theta|_{\infty}^{d_0-q+1} \cdot |\theta|_{\infty}^{-(q^3-1)/(q-1) + mq^2/(q-1)} < 1.
\end{equation}
Thus $\norm{E_2(x^m)} < 1$, and so $E_2(x^m)=0$.
\end{example}

\begin{example} \label{Ex:two}
If we assume instead the weaker conditions that $0 \leq m \leq q-2$ and $0 \leq d_0 \leq q$ (but still $0 \leq m + d_0 < \tfrac{3}{2} (q-1)$), then we find
\begin{equation}
  \cE_{\phi}(x^m,z) = x^m z + E_1(x^m)z^q,
\end{equation}
where it is possible that $E_1(x^m)$ is non-zero.  The derivation is similar to Example~\ref{Ex:one}, and the observation is that similar estimates to~\eqref{E:normS2}--\eqref{E:normalpha2S0} can be made with only small modifications.
\end{example}

We now consider the prospect of calculating $E_i(x^m)$ (in particular $E_1(x^m)$ as in Example~\ref{Ex:two}), and so we investigate computing $\mu(a)$.  For $f \in A_+$ irreducible of degree $d$, the Hasse invariant $H(\phi;f) \in A$ is defined to be the coefficient $\brac{f}{d}$ of $\tau^d$ in $\phi_f$, and utilizing results of Gekeler, Hsia, and Yu~\cite{Gekeler83}, \cite{Gekeler91}, \cite{HsiaYu00}, we observe that
\begin{equation} \label{E:mufHasse}
  \mu(f) = H(\phi;f) \bmod f = \brac{f}{d} \bmod f.
\end{equation}
That is, $\mu(f)$ is the remainder of $\brac{f}{d}$ upon division by $f$.  When $d=1$ and $f = \theta + c$, $c \in \FF_q$, then this implies
\begin{equation} \label{E:mudeg1}
  \mu(f) = g \bmod f = g(-c).
\end{equation}
When $d=2$ and $f = \theta^2 + c_1\theta + c_0$, $c_i \in \FF_q$, we have
\[
  \mu(f) = \bigl( (\theta^{q^2} + \theta +c_1 ) \Delta + g^{q+1} \bigr) \bmod f,
\]
and correspondingly the complexity of $\mu(f)$ grows with $d$, though it would be interesting to investigate how the methods of \cite{EP13}, \cite{EP14} might be applied to simplify it.

\begin{lemma} \label{L:S1beta}
Assume that $d_0=q$ and $0 \leq m \leq q-2$, and let $\beta = x^m$.  Suppose that $g =\sum_{i=0}^q b_i \theta^i$, $b_i \in \FF_q$.  Then
\[
  S_1(x^m) = \biggl( -\frac{g}{\theta^q-\theta} + b_q \biggr) x^{mq}
  - \sum_{\ell=1}^m \sum_{i=q-\ell}^{q-1} (-1)^i b_i \binom{m}{\ell} \binom{\ell-1}{q-1-i} \theta^{\ell-q+i} \cdot x^{mq - \ell(q-1)}.
\]
\end{lemma}

\begin{proof}
We first note that
\begin{align*}
S_1(x^m) &= \sum_{c \in \FF_q} \frac{\mu(\theta+c) ((\theta+c)x + x^q)^m}{\theta+c} \\
&= \sum_{c\in \FF_q} \sum_{i=0}^q \frac{b_i(-c)^i ((\theta+c)x + x^q)^m}{\theta+c} \\
&= \begin{aligned}[t]
  b_0 &\sum_{\ell=0}^m \binom{m}{\ell} x^{\ell + q(m-\ell)} \sum_{c \in\FF_q} (\theta+c)^{\ell-1} \\
  &{}+ \sum_{i=1}^q (-1)^i b_i \sum_{\ell=0}^m \binom{m}{\ell} x^{\ell+q(m-\ell)}
  \sum_{c \in \FF_q} c^i(\theta+c)^{\ell-1},
  \end{aligned}
\end{align*}
where the middle equality follows from~\eqref{E:mudeg1}.  It is well-known \cite[Ch.~5]{Thakur} that
\[
  \sum_{c \in \FF_q} (\theta+c)^{\ell-1} = \begin{cases}
  -\dfrac{1}{\theta^q-\theta} & \textup{if $\ell=0$,} \\
  0 & \textup{if $1 \leq \ell \leq q-1$,}
  \end{cases}
\]
and by direct calculation we see that for $i \geq 1$ and $\ell \geq 1$,
\[
  \sum_{c \in \FF_q} c^i (\theta+c)^{\ell-1} = -\sum_{\substack{s=0 \\ (q-1) \mid (s+i) }}^{\ell-1} \binom{\ell-1}{s} \theta^{\ell-1-s}.
\]
We then observe that for $1 \leq i \leq q-1$,
\begin{equation} \label{E:deg1}
\sum_{c\in \FF_q}\frac{c^i}{\theta+c}=(-1)^{i+1}\frac{\theta^i}{\theta^q-\theta},
\end{equation}
which follows from standard arguments (e.g., see the proof of \cite[Lem.~3.1]{EPetrov15}).  We note also that $0 \leq m \leq q-2$ and $1 \leq i \leq q$ imply that $1 \leq s+i \leq 2q-3 < 2(q-1)$, so $(q-1) \mid (s+i)$ if and only if $s = q-1-i$.  Taking this all together, we see that
\begin{align*}
S_1(x^m) = - \sum_{i=0}^{q-1} &\frac{b_i \theta^i}{\theta^q-\theta}\cdot x^{mq} - \frac{b_q \theta}{\theta^q-\theta} \cdot x^{mq} \\
&{}- \sum_{i=1}^{q-1} (-1)^i b_i \sum_{\ell=q-i} \binom{m}{\ell} \binom{\ell-1}{q-1-i} \theta^{\ell-q+i} \cdot x^{\ell+q(m-\ell)},
\end{align*}
from which the result follows.
\end{proof}

\begin{example}
By Lemma~\ref{L:S1beta}, if $g = \sum_{i=0}^q b_i \theta^i$ ($\deg g \leq q$), $\deg \Delta \leq q$, and $m=0$, then
\[
  S_1(1) = -\frac{g}{\theta^q-\theta} + b_q,
\]
and since $\alpha_1 = g/(\theta^q-\theta)$, we have
\[
  E_1(1) = S_1(1) + \alpha_1 S_0(1)^q = b_q.
\]
Therefore,
\begin{equation} \label{E:cE1}
  \cE_{\phi}(1,z) = z + b_q z^q.
\end{equation}
If we take $m=1$ and assume $q\geq 3$, then
\[
  S_1(x) = \biggl(-\frac{g}{\theta^q -\theta} + b_q \biggr) x^q - b_{q-1} x,
\]
and so
\[
  E_1(x) = S_1(x) + \alpha_1 S_0(x)^q = b_qx^q - b_{q-1}x.
\]
Thus
\begin{equation} \label{E:cEx}
  \cE_{\phi}(x,z) = xz + (b_q x^q - b_{q-1}x) z^q, \quad (q \geq 3).
\end{equation}
We note that the condition $m+q < \tfrac{3}{2}(q-1)$ may not be satisfied for $q \leq 5$, but in these cases \eqref{E:cE1} and \eqref{E:cEx} can be verified directly.
\end{example}

\begin{example}
We present now computations of special $L$-values for $\phi$ in terms of logarithms as in \S\ref{S:LogAlg}.  Assume that $g = \sum_{i=0}^q b_i\theta^i$ ($\deg g \leq q$) and that $\deg \Delta \leq q$.  By~\eqref{E:Lphibeta}, we have $L(\phi^{\vee},0) = \cL_{\phi}(1,z)|_{z=1}$, and so by~\eqref{E:cE1},
\begin{equation}
\exp_{\phi} \bigl( L(\phi^{\vee},0) \bigr) = \cE_{\phi}(1,z)|_{z=1} =  1 + b_q \in \FF_q.
\end{equation}
It is tempting to write ``$L(\phi^{\vee},0) = \log_{\phi}(1 + b_q)$,'' but $1+b_q$ may not be within the radius of convergence of $\log_\phi(z)$ (see~\cite[Cor.~4.2]{EP13}, \cite[Rem.~6.11]{EP14}).  Instead we will write
\[
L(\phi^{\vee},0) = \Log_{\phi}(1+b_q)
\]
to express that $L(\phi^{\vee},0)$ is a logarithm of $1+b_q$.

If we let $\chi : A \to \FF_q$ be the Dirichlet character defined by $\chi(a) = a(0)$ and we take $\zeta =\be(1/\theta) = (-\theta)^{1/(q-1)}$, then
\[
  \cL_{\phi}(x,z)\big|_{x=\zeta,z=1} = L(\phi^{\vee},\chi,0) \cdot \zeta.
\]
Then for $q \geq 3$, \eqref{E:cEx} implies
\begin{equation}
  \exp_{\phi} \bigl( L(\phi^{\vee},\chi,0) \cdot \zeta \bigr) = \cE_{\phi}(x,z)\big|_{x=\zeta,z=1} = (1-b_{q-1})\zeta + b_q\zeta^q, \quad (q \geq 3),
\end{equation}
and so using the convention for $\Log_\phi$ (versus $\log_\phi$) from the previous paragraph
\[
L(\phi^{\vee},\chi,0) = \frac{1}{\zeta} \cdot \Log_{\phi} \bigl(\zeta(1-b_{q-1} -b_q\theta)\bigr).
\]
\end{example}


\begin{thebibliography}{99}

\bibitem{And94} %
G. W. Anderson, \textit{Rank one elliptic $A$-modules and $A$-harmonic series}, Duke Math. J. \textbf{83} (1994), no. 3, 491--542.

\bibitem{And96} %
G. W. Anderson, \textit{Log-algebraicity of twisted $A$-harmonic series and special values of $L$-series in characteristic $p$}, J. Number Theory \textbf{60} (1996), no. 1, 165--209.

\bibitem{AnglesNgoDacTavares17} %
B. Angl\`{e}s, T. Ngo Dac, and F. Tavares Ribeiro, \textit{Stark units in positive characteristic}, Proc. Lond. Math. Soc. (3) \textbf{115} (2017), no.~4, 763--812.

\bibitem{AnglesPellarinTavares16} %
B. Angl\`{e}s, F. Pellarin, and F. Tavares Ribeiro, \textit{Arithmetic of positive characteristic $L$-series values in Tate algebras. With an appendix by F.~Demeslay}, Compos. Math. \textbf{152} (2016), no.~1, 1--61.

\bibitem{AnglesPellarinTavaresTAMS} %
B. Angl\`{e}s, F. Pellarin, and F. Tavares Ribeiro, \textit{Anderson-Stark units for $\mathbb{F}_q[\theta]$}, Trans. Amer. Math. Soc. (to appear).

\bibitem{AnglesTaelman15} %
B. Angl\`{e}s and L. Taelman, \textit{Arithmetic of characteristic $p$ special $L$-values. With an appendix by V.~Bosser}, Proc. Lond. Math. Soc. (3) \textbf{110} (2015), no.~4, 1000--1032.

\bibitem{AnglesTavares17} %
B. Angl\`{e}s and F. Tavares Ribeiro, \textit{Arithmetic of function field units}, Math. Ann. \textbf{367} (2017), no.~1--2, 501--579.

\bibitem{CP11} %
C.-Y. Chang and M. A. Papanikolas, \textit{Algebraic relations among periods and logarithms of rank~$2$ Drinfeld modules}, Amer. J. Math. \textbf{133} (2011), no.~2, 359--391.

\bibitem{CP12} %
C.-Y. Chang and M. A. Papanikolas, \textit{Algebraic independence of periods and logarithms of Drinfeld modules. With an appendix by B.~Conrad}, J. Amer. Math. Soc. \textbf{25} (2012), no.~1, 123--150.

\bibitem{EP13} %
A. El-Guindy and M. A. Papanikolas, \textit{Explicit formulas for Drinfeld modules and their periods}, J. Number Theory \textbf{133} (2013), no.~6, 1864--1886.

\bibitem{EP14} %
A. El-Guindy and M. A. Papanikolas, \textit{Identities for Anderson generating functions for Drinfeld modules}, Monatsh. Math. \textbf{173} (2014), no. 3--4, 471--493.

\bibitem{EPetrov15} %
A. El-Guindy and A. Petrov, \textit{On symmetric powers of $\tau$-recurrent sequences and deformations of Eisenstein series}, Proc. Amer. Math. Soc. \textbf{143} (2015), no. 8, 3303--3318.

\bibitem{Fang15} %
J. Fang, \textit{Equivariant special $L$-values of abelian $t$-modules}, arXiv:1503.07243, 2015.

\bibitem{Gardeyn02} %
F. Gardeyn, \textit{A Galois criterion for good reduction of $\tau$-sheaves}, J. Number Theory \textbf{97} (2002), no.~2 , 447--471.

\bibitem{Gekeler83} %
E.-U. Gekeler, \textit{Zur Arithmetik von Drinfeld-Moduln}, Math. Ann. \textbf{262} (1983), no. 2, 167--182.

\bibitem{Gekeler91} %
E.-U. Gekeler, \textit{On finite Drinfeld modules}, J. Algebra \textbf{141} (1991), no. 1, 187--203.

\bibitem{Goss83} %
D. Goss, \textit{On a new type of $L$-function for algebraic curves over finite fields}, Pacific J. Math. \textbf{105} (1983), no.~1, 143--181.

\bibitem{Goss92} %
D. Goss, \textit{$L$-series of $t$-motives and Drinfeld modules}, in: The Arithmetic of Function Fields (Columbus, OH, 1991), de Gruyter, Berlin, 1992, pp.~313--402.

\bibitem{Goss} %
D. Goss, \textit{Basic Structures of Function Field Arithmetic}, Springer-Verlag, Berlin, 1996.

\bibitem{GreenP} %
N. Green and M. A. Papanikolas, \textit{Special $L$-values and shtuka functions for Drinfeld modules on elliptic curves}, Res. Math. Sci. (to appear).

\bibitem{HsiaYu00} %
L.-C. Hsia and J. Yu, \textit{On characteristic polynomials of geometric Frobenius associated to Drinfeld modules}, Compositio Math. \textbf{122} (2000), no. 3, 261--280.

\bibitem{LP13} %
B. A. Lutes and M. A. Papanikolas, \textit{Algebraic independence of values of Goss $L$-functions at $s=1$}, J. Number Theory \textbf{133} (2013), no.~3, 1000--1011.

\bibitem{PLogAlg} %
M. A. Papanikolas, \textit{Log-algebraicity on tensor powers of the Carlitz module and special values of Goss $L$-functions}, in preparation.

\bibitem{Taelman09} %
L. Taelman, \textit{Special $L$-values of $t$-motives: a conjecture}, Int. Math. Res. Not. IMRN \textbf{2009} (2009), no.~16, 2957--2977.

\bibitem{Taelman10} %
L. Taelman, \textit{A Dirichlet unit theorem for Drinfeld modules}, Math. Ann. \textbf{348} (2010), no.~4, 899--907.

\bibitem{Taelman12} %
L. Taelman, \textit{Special $L$-values of Drinfeld modules}, Ann. of Math. (2) \textbf{175} (2012), no.~1, 369--391.

\bibitem{Thakur92} %
D. S. Thakur, \textit{Drinfeld modules and arithmetic in the function fields}, Internat. Math. Res. Notices \textbf{1992} (1992), no. 9, 185--197.

\bibitem{Thakur} %
D. S. Thakur, \textit{Function Field Arithmetic}, World Scientific Publishing, River Edge, NJ, 2004.

\bibitem{Washington} %
L. C. Washington, \textit{Introduction to Cyclotomic Fields}, Springer-Verlag, New York, 1982.

\bibitem {Yu86} %
J. Yu, \textit{Transcendence and Drinfeld modules}, Invent. Math.  \textbf{83} (1986),  no.~3, 507--517.

\bibitem {Yu97} %
J. Yu, \textit{Analytic homomorphisms into Drinfeld modules}, Ann. of Math. (2) \textbf{145} (1997), no.~2, 215--233.


\end{thebibliography}
\end{document}